\documentclass[10pt]{article}
\usepackage[utf8]{inputenc}
\usepackage[english]{babel}
\usepackage{amsfonts,amsmath,amssymb,amsthm}
\usepackage[colorlinks=true]{hyperref}
\usepackage[dvipsnames]{xcolor}
\usepackage{cases}
\usepackage{empheq}
\usepackage{cleveref}
\usepackage{bbm}
\usepackage{enumerate}
\usepackage{geometry}
\usepackage[title]{appendix}
\usepackage{color}
 \usepackage[shortlabels]{enumitem}
 \newtheorem{lemma}{\textit{Lemma}}
  \newtheorem{theorem}{\textit{Theorem}}
  \newtheorem{definition}{\textit{Definition}}
  \newtheorem{proposition}{\textit{Proposition}}
   \newtheorem{remark}{\textit{Remark}}
   
   \newtheorem{cor}{Corollary}
\def\E{\mathbb{E}}
\def\R{\mathbb{R}}
\def\P{\mathbb{P}}
\def\Q{\mathbb{Q}}
\def\N{\mathbb{N}}
\def\W{\mathbb{W}}

\def\1{\mathbbm{1}}
\def\F{\mathcal{F}}

\definecolor{MidnightBlue}{rgb}{0.1953125,0.1953125,0.796875}

\newcommand{\Pp}{\mathcal P}

\newcommand{\modif}[1]{#1}
\title{Mean--field limit of a particle approximation of the one-dimensional parabolic--parabolic Keller-Segel model without smoothing}
\author{Jean-Fran\c cois Jabir\footnote{National Research University 
Higher School of Economics, Shabolovka $28/11$, Moscow, 
Russian
Federation (jjabir@hse.ru).}, Denis Talay\footnote{TOSCA, Inria Sophia Antipolis-Méditerranée, 2004 Route des Lucioles, 06902 Valbonne, France (denis.talay@inria.fr, milica.tomasevic@inria.fr).} and Milica Toma\v{s}evi\'c \footnotemark[\value{footnote}] \footnote{Universit\'e C\^ote d'Azur, Campus Valrose, Batiment M, 28 Avenue de Valrose, 06108 Nice CEDEX 2, France. }}
\date{\today}
\date{\empty}
\begin{document}
\maketitle
\noindent
\textbf{Abstract:} In this work, we prove the well--posedness of a singularly interacting stochastic particle system and we establish propagation of chaos result towards the one-dimensional parabolic-parabolic Keller-Segel model.
\textbf{Classification:} 39A50, 60H30, 80C22.
%\subjclass[2010]{Primary: 39A50; Secondary: 60H30}; Tertiary: 80C22}
\section{Introduction}
%\subjclass[2010]{Primary: 39A50; Secondary: 60H30}; Tertiary: 80C22}

The standard $d$-dimensional parabolic--parabolic Keller--Segel model
for chemotaxis describes  the time evolution of the density $\rho_t$ of
a cell population and of the concentration $c_t$ of a chemical
attractant:
\begin{equation}
\label{KS}
\begin{cases}
&  \partial_t \rho(t,x)=\nabla \cdot (\frac{1}{2}\nabla \rho -  \chi\rho \nabla
c), \quad t>0, \ x \in \R^d, \\
 & \alpha \partial_t c(t,x) = \frac{1}{2} \triangle c  -\lambda c + \rho,  \quad
 t>0, \ x \in \R^d,\\
 & \rho(0,x)=\rho_0(x), c(0,x)=c_0(x),
\end{cases}
\end{equation}
\modif{for some parameters $\chi>0$, $\lambda\geq 0$ and $\alpha \geq
0$}. See e.g. Corrias {\it et al.} ~\cite{Corrias2014},
Perthame~\cite{perthame} and
references therein for theoretical results on this system of PDEs and
applications to \modif{b}iology. \modif{When $\alpha=0$, the system \eqref{KS} is parabolic--elliptic, and when $\alpha=1$~(or more generally, when $0<\alpha\leq 1$), the system is parabolic--parabolic.}

For the parabolic--elliptic version of the model with $d=2$, the first stochastic interpretation of this system is due to Ha\v{s}kovec and Schmeiser~\cite{haskovec-schmeiser} who analyze a particle system with
McKean--Vlasov interactions and Brownian noise. More precisely, as the
ideal interaction kernel should be strongly singular, they introduce a
kernel with a cut-off parameter and obtain the tightness of the
particle probability distributions w.r.t. the cut-off parameter and the
number of particles. They also obtain partial results in the direction
of the propagation of chaos. More recently, in the
subcritical case, that is, when the parameter~$\chi$ of the
parabolic--elliptic model is small enough, Fournier and
Jourdain~\cite{fournier-jourdain} obtain the well--posedness of a
particle system without cut-off. In addition, they obtain a consistency property which is weaker than the propagation of chaos. They also describe complex
behaviors of the particle system in the sub and super critical cases.
Cattiaux and P\'ed\`eches~\cite{cattiaux-pedeches} obtain the
well-posedness of this particle system without cut-off by using Dirichlet forms rather
than pathwise approximation techniques.

For a parabolic--parabolic version of the model with a smooth coupling
between $\rho_t$ and $c_t$, Budhiraja and Fan \cite{Budhiraja} study a
particle system with a smooth time integrated kernel and
prove it propagates chaos. Moreover, adding a forcing potential
term to the model, under a suitable convexity assumption, they obtain
uniform in time concentration
inequalities for the particle system and uniform in time error
estimates for a
numerical approximation of the limit non-linear process. % fin
%de rouge

For the pure parabolic--parabolic model without cut-off or
smoothing, in the one-dimensional case with $\alpha=1$, Talay and
Toma\v{s}evi\'c~\cite{Mi-De} have proved the well-posedness of PDE \eqref{KS} and of the following non-linear SDE:
\begin{equation}
\label{NLSDE}
\begin{cases}
& dX_t= b(t,X_t)dt +\Big\{\chi \int_0^t (K_{t-s}\star \rho_s)(X_t)ds\Big\} dt
+  dW_t, \quad t> 0, \\%Old 15 12: \quad t\leq  T, \\
& \rho_s(y)dy:=  \mathcal{L}( X_{s}),\quad X_0 \sim \rho_0(x)dx,
\end{cases}
\end{equation}
where $K_t(x):=e^{-\lambda t}\frac{\partial}{\partial x}(
\frac{1}{(2\pi t)^{1/2}}e^{-\frac{x^2}{2t}})$ and $b(t,x)=e^{-\lambda
t}  \frac{\partial}{\partial x} \E[ c_0(x+W_t)]$.
%Notice that this system is non classical in the sense that it involves
%all the past time marginals of the probability distribution of the
%solution.
%study a different type of stochastic
%representation of the parabolic--parabolic ($\alpha=1$) Keller--Segel model: we
%construct a McKean--Vlasov non-linear stochastic differential system
%with a singular interaction which is non classical in the sense that it
%involves all the past time marginals of the probability distribution of
%the solution. The system reads
%W.l.o.g. one may suppose that $\|\rho_0\|_{L^1(\R^d)}=1$.
%\begin{equation}
%\label{NLSDE}
%\begin{cases}
%& dX_t= b(t,X_t)dt +\Big\{\chi \int_0^t (K_{t-s}\star \rho_s)(X_t)ds\Big\} dt
%+  dW_t, \quad t> 0, \\%Old 15 12: \quad t\leq  T, \\
%& p_s(y)dy:=  \mathcal{L}( X_{s}),\quad X_0 \sim \rho_0(x)dx,
%\end{cases}
%\end{equation}
%where $K_t(x):=e^{-\lambda t}\nabla( \frac{1}{(2\pi t)^{d/2}}e^{-\frac{|x|^2}{2t}})$ and $b(t,x)=e^{-\lambda t}  \nabla \E[ c_0(x+\sqrt{2}W_t)]$.

Under the sole condition that the initial probability
law~$\mathcal{L}(X_0)$ has a density,
it is shown that the law $\mathcal{L}(X)$ uniquely solves a non-linear
martingale
problem and its time marginals have densities. These densities coupled
with a suitable transformation of them uniquely solve the
one--dimensional parabolic--parabolic Keller--Segel system without
cut-off. In Toma\v{s}evi\'c~\cite{Mi-De-2D}, additional techniques are
being developed for the two-dimensional version of~(\ref{NLSDE}).

The objective of this note is to analyze the particle system related to \eqref{NLSDE}. It inherits from the limit equation
that at each time $t>0$ each particle interacts in a singular way with
the past of all the other particles. We prove that the particle system
is well--posed and propagates chaos to the unique weak solution of
\eqref{NLSDE}. Compared to the stochastic
particle systems introduced for the parabolic--elliptic model, an
interesting fact occurs: the difficulties arising from the singular
interaction can now be resolved by using purely Brownian techniques
rather than by using Bessel processes. Due to the singular nature of
the kernel $K$, we need to introduce a partial Girsanov transform of
the $N$-particle system in order to obtain uniform in $N$ bounds for
moments of the corresponding exponential martingale. Our calculation is
based on the fact that the kernel $K$ is in $L^1(0,T;L^2(\R))$. We aim
to address in the close future the multi-dimensional Keller--Segel
particle system where the $L^1(0,T;L^2(\R^{d}))$-norm of the kernel is
infinite.

%In higher dimensions, $L^1(0,T;L^2(\R))$ norm of the kernel is
%infinite.

The paper is organized as follows. In Section~\ref{sec:main-results}
we state our two main results and comment our methodology.
In Section~\ref{sec:preliminaries} and Appendix we prove technical lemmas.
In Section~\ref{sec:propagation-chaos} we prove our main results.

In all the paper we denote by $C$ any positive real number independent of $N$.
% that the particle
%system is well--posed and propagates chaos to the solution
%of~\eqref{NLSDE}.
%=============================================================
% End of Denis' tex
\section{Main results}
\label{sec:main-results}

Our main results concern the well--posedness and propagation of chaos
of
%the following system of one--dimensional particles:
%existence and propagation of chaos toward \eqref{NLSDEsimple} of the weak solution to the following particle system:
%\begin{equation*}
%\begin{cases}
%&dX_t^{i,N}= b(t,X_t^{i,N})dt + \left\{\frac{\chi}{N} \sum_{j=1,j\neq
%i}^N \int_0^t K_{t-s}(X_t^{i,N}-X_s^{j,N})ds \mathbbm{1}_{\{X_t^i\neq
%X_t^j\}}  \right\}dt + dW_t^i,\\
%& X_0^i, \text{ i.i.d.} \text{ independent of } (W^i, 1\leq i \leq N),
%\end{cases}
%\end{equation*}
\begin{equation}
\label{PSsimple}
\begin{cases}
&dX_t^{i,N}= \left\{\frac{1}{N} \sum_{j=1,j\neq i}^N \int_0^t
K_{t-s}(X_t^{i,N}-X_s^{j,N})ds~\mathbbm{1}_{{\{X_t^{i,N}\neq
X_t^{j,N}\}}}  \right\}dt + dW_t^i, \\
& X_0^{i,N} \text{ i.i.d. and  independent of } W:=(W^i, 1\leq
i \leq N),
\end{cases}
\end{equation}
where %$T$ is an arbitrary time horizon, and the
$K_t(x)=\frac{-x}{\sqrt{2\pi}
t^{3/2}}e^{-\frac{x^2}{\modif{2}t}}$ and the $W^i$'s %Old: $(W^i_t)$
 are $N$ independent standard Brownian motions.
%Notice that the indicator $\mathbbm{1}_{\{X_t^i\neq X_t^j\}}$ in the drift removes the interactions of the $i^{th}$ particle with itself, {\color{blue}thus we could omit $j\neq i$ in the sum}.
It corresponds to~$\modif{\alpha=1}$, $\lambda =0$, $\chi =1$, and \modif{$c_0' \equiv 0$}.
It is easy to extend our methodology to~(\ref{NLSDE}) under the
hypotheses made in~\cite{Mi-De}.

\begin{theorem}
\label{PSexistence}
Given $0<T<\infty$ and $N \in \N$,
% for any initial distribution of $(X_0^{i,N}, i\leq N)$,
 there exists a weak solution $(\Omega,\F,
(\F_t;\,0\leq t\leq T),\Q^N, W, X^N)$ to the $N$-interacting particle
system \eqref{PSsimple} that satisfies, for any $1\leq i \leq N$,
%\forall 1\leq i\leq N, ~~~
\begin{equation}\label{DriftL2}
 \Q^N\left(\int_0^T\left(\frac{1}{N}\sum_{j=1,j\neq i}^N \int_0^t
 K_{t-s}(X_t^{i,N}-X_s^{j,N})ds \mathbbm{1}_{\{X_t^{i,N}\neq
 X_t^{j,N}\}}\right)^2\,dt<\infty\right)=1.
%Old 18 12
%\forall 1\leq i\leq N, ~~~
%\Q^N\left(\int_0^T\left|\frac{1}{N}\sum_{j=1,j\neq i}^N \int_0^t
%K_{t-s}(X_t^{i,N}-X_s^{j,N})ds \mathbbm{1}_{\{X_t^{i,N}\neq
%X_t^{j,N}\}}\right|\,dt<\infty\right)=1.
\end{equation}

%Given $0<T<\infty$ and $N \in \N$, for any initial distribution of $(X_0^i, i\leq N)$, there exists a unique weak solution to the $N$-interacting particle system \eqref{PSsimple}.
\end{theorem}
In view of Karatzas and Shreve ~ \cite[Chapter 5, Proposition 3.10]{KaratzasShreve}, one has the following uniqueness result:
%Old: In view of \cite[Chapter 5, Proposition 3.10]{KaratzasShreve}, one has the following uniqueness result:
\begin{cor}
Weak uniqueness holds in the class of weak solutions satisfying \eqref{DriftL2}.
\end{cor}
The construction of a weak solution to \eqref{PSsimple} involves
arguments used by Krylov  and R\"ockner \cite[Section $3$]{KryRoc-05}
to construct a weak solution to SDEs with singular drifts. It relies on
the Girsanov transform which removes all the drifts of
\eqref{PSsimple}.
%It is justified by the fact that exponential moments of  $\int_0^T\left(\int_0^t K_{t-s}(X^{i}_t-X^{j}_{s})\,ds\right)^2 dt$ are bounded whenever either $X^i$ or $X^j$ is a Brownian motion and $X^i$ and $X^j$ are independent.
\begin{remark}
%In \eqref{PSsimple}, the indicator of $\{X_t^{i,N}\neq X_t^{j,N}\}$
%makes it clear that the drift coefficient is well defined.
Our construction shows that the law of the particle system is
equivalent to Wiener's measure. Thus, a.s. the set
$\{t\leq T, X_t^{i,N} = X_t^{j,N}\}$ has Lebesgue measure zero.
\end{remark}
Our second main theorem concerns the propagation of chaos of the system \eqref{PSsimple}. \modif{Before we proceed to its statement}\modif{, we need to define the non-linear martingale problem (MPKS) associated to the non-linear SDE:
\begin{equation}
\label{NLSDEsimple}
\begin{cases}
& dX_t= \Big\{\int_0^t (K_{t-s}\star \rho_s)(X_t)ds\Big\} dt
+ dW_t, \quad t\leq  T, \\
& \rho_s(y)dy:=  \mathcal{L}( X_{s}),\quad X_0 \sim \rho_0(x)dx.
\end{cases}
\end{equation}
 For any measurable space $E$ we
denote by $\mathcal{P}(E)$ the set of probability measures on~$E$.}

\modif{
\begin{definition}
$\Q \in \mathcal{P}(C[0,T];\R)$ is a solution to (MPKS) if:
\begin{enumerate}[(i)]
\item $\Q_0(dx) = \rho_0(x)~dx$;
\item For any $t\in (0,T]$, the one dimensional time marginal~$\Q_t$
of $\Q$ has ~a density~ $\rho_t$ w.r.t. Lebesgue measure on $\R$ which belongs to $L^2(\R)$ ~and~satisfies~
%Old: For any $t\in (0,T]$, the one dimensional time marginals  $\Q_t$
%of $\Q$ have densities $\rho_t$ w.r.t. Lebesgue measure on $\R$ which
%belong to $L^2(\R)$ for any $0<t\leq T$ and satisfy
$$ \exists C_T,~~\forall\, 0<t\leq T, ~~~ \Vert \rho_t\Vert_{L^2(\R)}
\leq \frac{C_{T}}{t^{1/4}}; $$
% ~~~ \text{ and } ~~~  \|\rho_t\|_{L^\infty(\R)}\leq \frac{C_T}{\sqrt{t}}
%Old 15 12:$$\|p_t\|_2 \leq \frac{C_{T}}{t^{1/4}},$$
\item Denoting by $(x(t);\,t\leq T)$ the canonical process of
$C([0,T];\R)$, we have: For any $f \in C_b^2(\R)$, the process defined
by
$$M_t:=f(x(t))-f(x(0))-\int_0^t \left(\Big(\int_0^s \int K_{s-
r}(x(s)-y) \rho_r(y)dydr\Big)f'(x(s))+\frac{1}{2} f''(x(s))\right)\,ds$$
%Old: $$M_t:=f(x(t))-f(x(0))-\int_0^t \left(\Big(\int_0^s \int K_{s-
%r}(x(s)-y) \rho_r(y)dydr\Big)\frac{\partial f}{\partial
%x}(x(s))+\frac{1}{2} \frac{\partial^2f}{\partial x^2}(x(s))\right)\,ds$$
is a $\Q$-martingale w.r.t. the canonical filtration.
\end{enumerate}
\end{definition}
}

%The following theorem comes from~\cite{Mi-De}.
\modif{In \cite{Mi-De}, the authors prove that (MPKS) admits a unique solution and that a suitable notion of weak solution to \eqref{NLSDEsimple}
is equivalent to the notion of solution to (MPKS)}.
%The global wellposedness of \eqref{NLSDE} is obtained in \cite{Mi-De}, given any finite time horizon $T$, $\chi>0$, and under the general assumption that $\rho_0 \in L^1(\R)$.
%
\begin{theorem}
\label{PSPropaChaos} Assume that \modif{the $X_0^{i,N}$'s are i.i.d. and}~ that the initial distribution of $X_0^{1,N}$
has a density $\rho_0$. %Old: for all $i\leq N$.
The empirical measure
$\mu^N=\frac{1}{N}\sum_{i=1}^N\delta_{X^{i,N}}$ of \eqref{PSsimple}
converges in the weak sense, when $N \to \infty$,  to the
unique weak solution of \eqref{NLSDEsimple}.

\end{theorem}

To prove the tightness and weak convergence of $\mu^N$, we use a
Girsanov transform which removes a fixed small number of the drifts of
\eqref{PSsimple} rather than all the drifts. This trick, which may be
useful for other singular interactions, allows us to get uniform
w.r.t.~$N$ bounds for the needed Girsanov exponential martingales.
\section{Preliminaries}
\label{sec:preliminaries}

On the path space define the functional
% $F_t$ as \begin{equation*}
%\label{Ft}
$ F_{t}(x,\widehat{x})= \left(\int_0^t
K_{t-s}(x_t-\widehat{x}_s)ds~\1_{\{x_t\neq \widehat{x}_t\}} \right)^2
$,
%\end{equation*}
where $(x,\widehat{x}) \in C([0,T];\R) \times C([0,T];\R)$.
The objective of this section is to show that $\int_0^T F_t(w,Y)dt$ has
finite exponential moments when $w$ is a Brownian motion and $Y$ is a
process independent of $w$. The following key property of the kernel
$K_t$ will be used:
\modif{
\begin{equation}
\label{l2normKernel}
\|K_{t}\|_{L^p(\R)}=\left(C\int_0^\infty
\frac{z^p}{t^{p-1/2}}e^{-\frac{pz^2}{2}}dz\right)^{1/p}=
\frac{C_p}{t^{1-1/2p}},\,1\leq p<\infty.
%Old:
%\Vert K_t\Vert_{L^2(\R)}
%=\left(\int_\R \frac{x^2}{t^2}g^2_t(x)\,dx\right)^{1/2}
%=\left(C\int_0^\infty
%\frac{z^2}{t^{3/2}}e^{-\jeff{z^2}}dz\right)^{1/2}=\frac{C}{t^{3/4}}.
%%\Vert K_t\Vert_{L^2(\R)}
%%%=\left(\int_\R \frac{x^2}{t^2}g^2_t(x)\,dx\right)^{1/2}
%%=\left(C\int_0^\infty
%%\frac{z^2}{t^{3/2}}e^{-\frac{z^2}{2}}dz\right)^{1/2}=\frac{C}{t^{3/4}}.
\end{equation}
}
We will proceed as in the proof of the local Novikov Condition
(see \cite[Chapter 3, Corollary 5.14]{KaratzasShreve})~
%Old: (see Karatzas and Shreve \cite[Chapter 3, Corollary 5.14]{KaratzasShreve}) <--- J.-F.: The modification was due to the fact that the book of Karatzas and Shreve was previously mentioned; so the rule "mention the name of the authors only the first time the reference is used, and don't repeat the name of the authors after"... .
by localizing on small intervals of time.
%Thus, we start with the following lemmas. In all the sequel $C>0$
%denotes any universal constant that may change from line to line.
\begin{lemma}
\label{MainLemmas1}
%Let $T>0$. Let $(\Omega,\F, (\F_t;\,0\leq t\leq T),\P)$ be %Old 15 12:
%%$(\P, \Omega, \F_t)$
% a filtered probability space equipped with a Brownian motion $W$.%Old
%%15 12: $\mu_0$.
Let $w:=(w_t)$ be a $(\mathcal{G}_t)$-Brownian motion with an arbitrary
initial distribution $\mu_0$ on some probability space equipped with a
probability measure $\mathbb{P}$ and a filtration~$(\mathcal{G}_t)$.
There exists a universal real number $C_0>0$ such that
$$ \forall x\in C([0,T];\R),~~\forall 0\leq t_1\leq t_2\leq T,~~
  \int_{t_1}^{t_2}\E^{\mathcal{G}_{t_1}}_\P F_t(w,x)dt
\leq C_0
\sqrt{T} \sqrt{t_2-t_1} .$$
\end{lemma}
%\E^{\mathcal{G}_{t_1}}_\P \left( \int_{t_1}^{t_2} F_t(w,x)dt\right)
%\leq C_0
%\sqrt{T} \sqrt{t_2-t_1}
\begin{proof}
By the definition of $F$,
%Fubini-Tonelli theorem and the definition of $F$,
\begin{equation}
\label{lemma1:1}
 \int_{t_1}^{t_2}\E^{\mathcal{G}_{t_1}}_\P F_t(w,x)dt\leq
\int_{t_1}^{t_2} \int_0^t\int_0^t \E^{\mathcal{G}_{t_1}}_\P
|K_{t-s}(w_t-x_s)K_{t-u}(w_t-x_u)|dsdudt.
\end{equation}
\modif{Let $g_t(x):=\frac{1}{\sqrt{2 \pi t}}e^{-\frac{x^2}{2t}}$.}~ In view of \eqref{l2normKernel}, one has
$$
%\sqrt{\E^{\mathcal{G}_{t_1}}_\P (K^2_{t-s}(w_t-x_s))}=
\sqrt{\int K^2_{t-s}(y+w_{t_1} -x_s))g_{t-t_1}(y)dy}\leq C\frac{\|
K_{t-s}\|_{L^2(\R)}}{(t-t_1)^{1/4}} \leq
\frac{C}{(t-s)^{3/4}(t-t_1)^{1/4}}.$$
%Old 18 12
%$$\sqrt{\E^{\mathcal{G}_{t_1}}_\P (K^2_{t-s}(w_t-x_s))}= \sqrt{\int
%K^2_{t-s}(y+w_{t_1} -x_s))g_{t-t_1}(y)dy}\leq \frac{C
%\|K_{t-s}\|_2}{(t-t_1)^{1/4}} \leq
%\frac{C}{(t-s)^{3/4}(t-t_1)^{1/4}}.$$
%
%Applying Cauchy-Schwarz inequality,
%$$\E^{\mathcal{G}_{t_1}}_\P |K_{t-s}(w_t-x_s)K_{t-u}(w_t-x_u)|\leq
%\sqrt{\E^{\mathcal{G}_{t_1}}_\P (K_{t-s}(w_t-x_s))^2}
%\sqrt{\E^{\mathcal{G}_{t_1}}_\P (K_{t-u}(w_t-x_u))^2}.$$
%we use $w_t=w_t-w_{t_1}+w_{t_1}$ and get
Here we used that the density of $w_t-w_{t_1}$ is bounded by
$\frac{C}{\sqrt{t-t_1}} $.
%Therefore,
%$$\E^{\mathcal{G}_{t_1}}_\P |K_{t-s}(w_t-x_s)K_{t-u}(w_t-x_u)|\leq
%\frac{C}{(t-u)^{3/4}(t-s)^{3/4}\sqrt{t-t_1}}$$
Coming back to \eqref{lemma1:1},
%Cauchy-Schwarz inequality leads to
$$ \int_{t_1}^{t_2}\E^{\mathcal{G}_{t_1}}_\P F_t(w,x)dt \leq
\int_{t_1}^{t_2}\frac{C}{\sqrt{t-t_1}}
\int_0^t\int_0^t\frac{1}{(t-s)^{3/4}(t-u)^{3/4}}dsdudt=\int_{t_1}^{t_2}
 \frac{C\sqrt{t}}{\sqrt{t-t_1}}dt. $$

%Thus, we proved that
%$$\E^{\F^W_{t_1}}_\P \left( \int_{t_1}^{t_2} F_t(W,x)dt\right)\leq C \sqrt{T} \sqrt{t_2-t_1}.$$
\end{proof}
%\begin{lemma}
%\label{Mainlemmas2}
%Same assumptions as in Lemma \ref{MainLemmas1}. Let $C_0$ be as in Lemma \ref{MainLemmas1}. Then,
%$$\forall \ 0\leq T_1\leq T_2 \leq T, \quad \quad  \E^{\F^W_{T_1}}_{\P}\left[\left(\int_{T_1}^{T_2} F_t(W,x)dt\right)^{k}\right]\leq C_0^k k!(\sqrt{T} \sqrt{T_2-T_1})^{k}.$$
%\end{lemma}
%\begin{proof}
%
%\end{proof}
\begin{lemma}
\label{Mainlemmas3}
Same assumptions as in Lemma \ref{MainLemmas1}. Let $C_0$ be as in
Lemma \ref{MainLemmas1}. For any $\kappa>0$, there exists $C(T,\kappa)$ \modif{independent of $\mu_0$}~
such that, for any $ 0\leq T_1\leq T_2 \leq T$ satisfying $  T_2-T_1 <
\frac{1}{C_0^2T\kappa^2} $,
$$\forall x\in C([0,T];\R),~~\E^{\mathcal{G}_{T_1}}_\P \left[\exp \left\{\kappa \int_{T_1}^{T_2}
F_t(w,x)dt\right\}\right]\leq C(T,\kappa).$$
\end{lemma}
\begin{proof}
We adapt the proof of Khasminskii's lemma in Simon~\cite{BSimon}. Admit
for a while we have shown that there exists a constant $C(\kappa, T)$
such that for any $M \in \N$
\begin{equation}
\label{prel:star}
\sum_{k=1}^M \frac{\kappa^k}{k!} \E^{\mathcal{G}_{T_1}}_\P \left(
\int_{T_1}^{T_2} F_t(w,x)dt\right)^{k}\leq C(T,\kappa),
\end{equation}
provided that $T_2-T_1 < \frac{1}{C_0^2T\kappa^2}$. The desired result
then follows from Fatou's lemma.

We now prove \eqref{prel:star}.
%Start with the equality
%\begin{align*}
%\left(\int_{T_1}^{T_2} F_t(W,x)dt\right)^{k}=k! \int_{T_1}^{T_2}
%F_{t_1}(W,x)\int_{t_1}^{T_2} F_{t_2}(W,x)\int_{t_2}^{T_2}\cdots
%  \int_{t_{k-2}}^{T_2}F_{t_{k-1}}(W,x)\int_{t_{k-1}}^{T_2}
%F_{t_k}(W,x)\,dt_k\,dt_{k-1}\cdots\,dt_2 \,dt_1.
%\end{align*}
%%We will use the tower property of conditional expectation and
%%condition one more time with respect to $\F^W_{t_{k-1}}$. Then, we
%%will
%%use Lemma \ref{MainLemmas1} and will repeat this procedure $k$ times
%%in
%%order to arrive to the desired result.\\\\
By the tower property of conditional expectation,
\begin{align*}
& \E^{\mathcal{G}_{T_1}}_{\P}\left[\left(\int_{T_1}^{T_2}
F_t(w,x)dt\right)^{k}\right]= k!
\int_{T_1}^{T_2}\int_{t_1}^{T_2} \int_{t_2}^{T_2}\cdots \int_{t_{k-2}}^{T_2}  \int_{t_{k-1}}^{T_2} \E^{\mathcal{G}_{T_1}}_{\P} \Big[ F_{t_1}(w,x)
F_{t_2}(w,x)\\
&
\times \dots\times F_{t_{k-1}}(w,x)\left(
 \E^{\mathcal{G}_{t_{k-1}}}_{\P} F_{t_k}(w,x) \right)\Big]\,dt_k \,dt_{k-1}\cdots\,dt_2
\,dt_1 .
\end{align*}
%\begin{align*}
%& \E^{\mathcal{G}_{T_1}}_{\P}\left[\left(\int_{T_1}^{T_2}
%F_t(w,x)dt\right)^{k}\right]= k!\E^{\mathcal{G}_{T_1}}_{\P} \Big[
%\int_{T_1}^{T_2} F_{t_1}(w,x)\int_{t_1}^{T_2}
%F_{t_2}(w,x)\int_{t_2}^{T_2}\cdots\\
%& \cdots
%\int_{t_{k-2}}^{T_2}F_{t_{k-1}}(w,x)\left(\E^{\mathcal{G}_{t_{k-1}}}_{\P}
% \int_{t_{k-1}}^{T_2} F_{t_k}(w,x)\,dt_k \right)\,dt_{k-1}\cdots\,dt_2
%\,dt_1 \Big].
%\end{align*}
In view of Lemma \ref{MainLemmas1},
$$\int_{t_{k-1}}^{T_2} \E^{\mathcal{G}_{t_{k-1}}}_{\P}
F_{t_k}(w,x)\,dt_k \leq C_0\sqrt{T} \sqrt{T_2-t_{k-1}}\leq C_0\sqrt{T}
\sqrt{T_2-T_1}.$$
Therefore, by Fubini's theorem,
\begin{align*}
&\E^{\mathcal{G}_{T_1}}_{\P}\left[\left(\int_{T_1}^{T_2}
F_t(w,x)dt\right)^{k}\right]\leq k!  C_0\sqrt{T} \sqrt{T_2-T_1}
\int_{T_1}^{T_2} \int_{t_1}^{T_2}\int_{t_2}^{T_2} \cdots  \int_{t_{k-2}}^{T_2} \E^{\mathcal{G}_{T_1}}_{\P}\Big[
F_{t_1}(w,x) F_{t_2}(w,x) \\
&\times \dots\times F_{t_{k-1}}(w,x)\Big]dt_{k-1}\cdots\,dt_2
\,dt_1.
\end{align*}
Now we repeatedly condition with respect to \modif{$\mathcal{G}_{t_{k-i}}$ ($i\geq 2$)}~ and
combine Lemma \ref{MainLemmas1} with Fubini's theorem. It comes:
%Repeating this procedure $k-1$ times,
\begin{equation*}
%\begin{split}
\E^{\mathcal{G}_{T_1}}_{\P}\left(\int_{T_1}^{T_2}
F_t(w,x)dt\right)^{k}
\leq k!  (C_0\sqrt{T} \sqrt{T_2-T_1})^{k-1}
\int_{T_1}^{T_2}\E^{\mathcal{G}_{T_1}}_{\P}  F_{t_1}(w,x)dt_1
%In view of Lemma \ref{MainLemmas1},
%\E^{\F^W_{T_1}}_{\P}\left[\left(\int_{T_1}^{T_2}
%F_t(W,x)dt\right)^{k}\right]
\leq k!  (C_0\sqrt{T} \sqrt{T_2-T_1})^{k}.
%\end{split}
\end{equation*}
%We expand the exponential,
%$$\E^{\F^W_{T_1}}_\P \left[\exp \left\{\kappa \int_{T_1}^{T_2} F_t(W,x)dt\right\}\right]=\E^{\F^W_{T_1}}_\P \sum_{k=1}^\infty \frac{\kappa^k}{k!} \left( \int_{T_1}^{T_2} F_t(W^1,W^2)dt\right)^{k} $$
%By Fatou's lemma we exchange the sum and the expectation.\\\\
%Applying Lemma \ref{Mainlemmas3},
%$$\E^{\F^W_{T_1}}_\P\left( \exp \{\kappa \int_{T_1}^{T_2} F_t(W,x)dt\} \right)\leq \sum_{k=1}^\infty \kappa^k (C\sqrt{T} \sqrt{T_2-T_1})^{k}, $$
%where $C$ is a deterministic constant, independent of $\kappa$ and $T$. Thus, the claim holds, provided that $T_2-T_1 < \frac{1}{C^2T\kappa^2}$.
Thus, \eqref{prel:star} is satisfied provided that $T_2-T_1 < \frac{1}{C_0^2T\kappa^2}$.
\end{proof}

\begin{proposition}
\label{est1}
 Let $T>0$. \modif{Same assumptions as in Lemma \ref{MainLemmas1}.
%Let $(\Omega,\F,(\F_t;\,0\leq t\leq T),\P)$ be %Old 15 12: $(\P,
%%%%\Omega, \F_t)$
Suppose that the
filtered probability space is rich enough to support a continuous process $Y$ independent
%an $(\mathcal{G}_t)$-adapted
 of~$(w_t)$.}~  For any
 $\alpha>0$,
\begin{equation*}
\E_{\P}\left[ \exp\left\{ \alpha \int_0^T F_t(w,Y)dt\right\}\right]
\leq C(T, \alpha),
\end{equation*}
where $C(T, \alpha)$ depends only on $T$ and $\alpha$, but
does neither depend on the law~$\mathcal{L}(Y)$ \modif{nor of $\mu_0$}.
%the law of the process $X$.
\end{proposition}
\begin{proof}
%[Proof of Proposition \ref{est1}]
Observe that
\begin{align}
\label{est1:1}
&\E_{\P} \exp\left\{ \alpha \int_0^T F_t(w,Y)dt\right\} =
\int_{C([0,T];\R)} \E_{\P}\exp\left\{ \alpha
\int_0^TF_t(w,x)dt\right\}  \P^{Y}(dx).
\end{align}
Set  $\delta: = \frac{1}{2C_0^2 T \alpha^2} \wedge T$, where $C_0$ is
as in Lemma \ref{MainLemmas1}. Set $n:= \left[ \frac{T}{\delta}
\right]$.
%Split the integral
%\[
%\int_{0}^{T}F_t(W,x)\,dt=\sum_{m=0}^{n}\int_{(T-(m+1)\delta)\vee 0 }^{T-m\delta}F_t(W,x)\,dt.
%\]
  Then,
 \begin{align*}
 &\exp\left\{\alpha \int_{0}^{T}F_t(w,x)dt\right\}
 ={\displaystyle\prod_{m=0}^{n}} \exp\left\{\alpha
\int_{(T-(m+1)\delta)\vee 0}^{T-m\delta}F_t(w,x)\,dt\right\}.
\end{align*}
Condition the right-hand side by $\mathcal{G}_{(T-\delta)\vee
0}$.
%\begin{align*}
%&\E_{\P}\left[ \exp\left\{ \alpha \int_0^T F_t(W,x)dt\right\}\right]
%\leq\E_{\P} \left[ {\displaystyle\prod_{m=1}^{n}} \exp\left\{\alpha
%\int_{(T-(m+1)\delta)\vee 0}^{T-m\delta}F_t(W,x)\,dt\right\}
%\E_\P^{\F^W_{T-\delta}}\left(\exp\{\alpha\int_{(T-\delta)\vee 0}^T
%F_t(W,x)dt\}\right) \right].
%\end{align*}
Notice that $\delta$ is small enough to be in the setting
of Lemma~\ref{Mainlemmas3}. Thus,
\begin{align*}
&\E_{\P} \exp\left\{ \alpha \int_0^T
F_t(w,x)dt\right\} \leq C(T,\alpha)\E_{\P}
{\displaystyle\prod_{m=1}^{n}} \exp\left\{\kappa N
\int_{(T-(m+1)\delta)\vee 0}^{T-m\delta}F_t(w,x)\,dt\right\}.
\end{align*}
Successively, conditioning by $\mathcal{G}_{(T-(m+1))\vee 0}$ for $
m=1, 2,
\dots n$ and using Lemma \ref{Mainlemmas3},
$$\E_{\P} \exp\left\{ \alpha \int_0^T
F_t(w,x)dt\right\} \leq C^n(T,\alpha)\E_{\P}
\exp\left\{\int_0^{(T-n\delta)\vee 0} F_t(w,x)dt\right\}
\leq  C(T,\alpha).$$
The proof is completed by plugging the preceding estimate
into~\eqref{est1:1}.
%\begin{align*}
%&\E_{\P}\left[ \exp\left\{ \alpha \int_0^T \left(\int_0^t K_{t-s}(W_t-X_s)ds\1_{\{W_t\neq X_t\}} \right)^2dt\right\}\right] \\
%& = \int_{C(0,T)} \E_{\P}\left[ \exp\left\{ \alpha \int_0^T \left(\int_0^t K_{t-s}(W_t-x_s)ds\1_{\{W_t\neq x_t\}} \right)^2dt\right\}\right]  \P^{X}(dx)\leq C(T, \alpha).
%\end{align*}
\end{proof}
\section{Existence of the particle system and  propagation of chaos}
\label{sec:propagation-chaos}
\subsection{Existence: Proof of Theorem \ref{PSexistence} }
We start from a probability space $(\Omega,\F,(\F_t;\,0\leq t\leq
T),\W)$ on which are defined an $N$ - dimensional Brownian motion
$W=(W^1, \dots, W^N)$ \modif{and the random variables $X_0^{i,N}$(see \eqref{PSsimple})}. Set $\bar{X}_t^{i,N} := X_0^{i,N} + W_t^i
~(t\leq T)$ and $\bar{X}:=(\bar{X}^{i,N},1\leq i\leq N)$.
Denote the drift terms in \eqref{PSsimple} by  $b^{i,N}_t(x)$, $x\in C([0,T];\R)^N$,
and the vector of all the drifts as $B_t^N(x)=(b^{1,N}_t(x), \dots, b^{N,N}_t(x))$.
For a fixed $N \in \N$, consider
$$Z_T^N:= \exp \left\{ \int_0^T B_t^N(\bar{X}) \cdot dW_t - \frac{1}{2}\int_0^T \left| B_t^N(\bar{X}) \right|^2dt \right\}.$$
To prove Theorem \ref{PSexistence}, it suffices to prove the following Novikov condition holds true
(see e.g. \cite[Chapter 3, Proposition 5.13]{KaratzasShreve}):
%Old 18 12
%(see e.g. \cite[Chapter 4, Section 4]{IkeWat}).
\begin{proposition}
\label{PS:novikovNpart}
For any $T>0$, $N \geq 1$, $\kappa >0$, there exists $C(T,N, \kappa)$ such that
\begin{equation} \label{ineq:exp-G-N}
\E_\W \left( \exp\left\{ \kappa \int_0^T |B_t^N(\bar{X})|^2dt\right\} \right)\leq
C(T,N, \kappa).
\end{equation}
\end{proposition}
\begin{proof}
Drop the index $N$ for simplicity. Using the definition of
$(B^N_t)$ and Jensen's inequality one has
%\begin{multline*}
%\E_{\W}\left[\exp\left\{\kappa\int_{0}^{T}
%\left|B_t^N(\bar{X})\right|^2\,dt\right\}\right]
%%&=\E_{\W}\left[\exp\left\{\kappa\sum_{i=1}^{N}\int_{0}^{T}\left|\frac{1}{N}\sum_{j=1}^N
%% \int_0^tK_{t-s}(\bar{X}_t^i-\bar{X}_s^j)\,ds\1_{\{\bar{X}^{i}_t\neq
%%\bar{X}^{j}_t\}}\right|^2\,dt\right\}\right]
%  \\
%=\E_{\W}\left[\exp\left\{\frac{1}{N}\sum_{i=1}^{N}\int_{0}^{T}\left(\kappa
% N\left|\frac{1}{N}\sum_{j=1}^N
%\int_0^tK_{t-s}(\bar{X}_t^i-\bar{X}_s^j)\,ds\1_{\{\bar{X}^{i}_t\neq
%\bar{X}^{j}_t\}}\right|^2\right)\,dt\right\}\right].
%  \end{multline*}
  \begin{align*}
  &\E_{\W}\left[\exp\left\{\kappa\int_{0}^{T}\left|B_t^N(\bar{X})\right|^2\,dt\right\}\right]\leq\E_{\W}\left[\exp\left\{\frac{1}{N}\sum_{i=1}^{N}\frac{1}{N}\sum_{j=1}^N \int_{0}^{T}\kappa NF_t(\bar{X}^i,\bar{X}^j)\,dt\right\}\right],
\end{align*}
from which we deduce
%Old 15 12: from which
$$   \E_{\W}\left[\exp\left\{\kappa\int_{0}^{T}\left|B_t^N(\bar{X})\right|^2\,dt\right\}\right]\leq \frac{1}{N}\sum_{i=1}^{N}\frac{1}{N}\sum_{j=1}^{N}
\E_{\W}\left[\exp\left\{\kappa
N\int_{0}^{T}F_t(\bar{X}^i,\bar{X}^j)\,dt\right\}\right].$$
As the $\bar{X}^{i}$'s %Old: $\bar{X}^{i}$'s WE DROPED THE INDEX N!!!!!
are independent Brownian motions,
%\jeff{starting from $\bar{X}^{i,N}_0$}, Initial cond plays no role
we are in a
position to use Proposition~\ref{est1}. This concludes the proof.

\end{proof}

%\begin{proof}[Proof of Theorem \ref{PSexistence}]
%Using the Proposition \ref{PS:novikovNpart} for $\kappa=\frac{1}{2}$, we get that the Novikov condition is satisfied for the drift $(B_t)_{ t\leq T}$. Thus,
%$$Z_T^N:= \exp \left\{ \int_0^T B_t \cdot dW_t - \frac{1}{2}\int_0^T \left| B_t \right|^2dt \right\}$$
%is a martingale under $\Q$. We define the probability measure $\P^N$ by $\frac{d\P^N}{d\Q}=Z_T$. By Girsanov theorem there exists an $N$ dimensional Brownian motion $(W)$  on $\P^N$ and an $N$ dimensional process $(X)$ such that \eqref{PSsimple} is satisfied. Thus, we have constructed a weak solution to \eqref{PSsimple}.
%\end{proof}
\subsection{Girsanov transform for $1\leq r<N$ particles}
\label{sec:GirsTrK}

In the proof of Theorem \ref{PSexistence} we used~\eqref{l2normKernel}
and a Girsanov transform. However, the right-hand side
of~(\ref{ineq:exp-G-N}) goes to infinity with $N$.
Thus, Proposition~\ref{PS:novikovNpart} cannot be used to prove
the tightness and propagation of chaos of the particle system. We
instead define an intermediate particle system. For any integer $1\leq
r<N$, proceeding as in the proof of Theorem~\ref{PSexistence}  one gets
the existence of a weak solution on $[0,T]$ to
%Old 15 12: a unique weak
%solution to
\begin{equation}
\label{PS_Kdrift}
\begin{cases}
&d\widehat{X}_t^{l,N}=  dW_t^l, \quad 1\leq l\leq r,\\
& d \widehat{X}_t^{i,N}=  \left\{\frac{1}{N} \sum_{j=r+1}^N \int_0^t
K_{t-s}(\widehat{X}_t^{i,N}-\widehat{X}_s^{j,N})ds
~\mathbbm{1}_{\{\widehat{X}_t^{i,N}\neq \widehat{X}_t^{j,N}\}} \right\}
dt + dW_t^i, \quad r+1\leq i\leq N,\\
& \widehat{X}_0^{i,N} %Old: X}_0^{i,N}
 \text{  i.i.d. and independent of }(W):=(W^i, 1\leq i\leq N).
\end{cases}
\end{equation}
Below we set $\hat{X}:=(\hat{X}^{i,N}, 1\leq i\leq N)$ and we denote by
$\Q^{r,N}$ the probability measure under which $\hat{X}$ is well
defined.
% \jeff{the probability measure underlying \eqref{PS_Kdrift} by $\Q^{k,N}$}~
%Old: so constructed probability measure by $\Q^{k,N}$
Notice that
$(\widehat{X}^{l,N}, 1\leq l\leq r)$ is independent of
$(\widehat{X}^{i,N}, r+1\leq i\leq N)$.  We now study the exponential
local martingale associated to the change of drift between
\eqref{PSsimple} and \eqref{PS_Kdrift}. For $x\in C([0,T];\R)^N$ set
\begin{multline*}
 \beta^{(r)}_t(x):= \Big(b_t^{1,N}(x), \dots, b_t^{r,N}(x), \frac{1}{N}
 \sum_{i=1}^r \int_0^t K_{t-s}(x_t^{r+1}-x^{i}_s)ds ~
 \1_{\{x_t^{r+1}\neq x^{i}_t\}}, \dots, \\
\frac{1}{N} \sum_{i=1}^r \int_0^t K_{t-s}(x_t^{N}-x^{i}_s)ds ~
\1_{\{x_t^{N}\neq x^{i}_t\}} \Big) .
\end{multline*}
In the sequel we will need uniform w.r.t~$N$ bounds for moments of
\begin{equation} \label{def:Z^(k)}
Z_T^{(r)}:= \exp\left\{ -\int_0^T \beta^{(r)}_t(\widehat{X}) \cdot
dW_t -\frac{1}{2}\int_0^T| \beta^{(r)}_t(\widehat{X})|^2 dt \right\}.
\end{equation}
% Then,
%$$= \exp \left\{- \int_0^T \beta^{(k)}_t \cdot dW_t
%-\frac{1}{2}\int_0^T\| \beta^{(k)}_t\|^2 dt \right\}.$$

\begin{proposition}
\label{GirsanovKparticles}
For any $T>0$, $\gamma>0$ and $r\geq 1$ there exists  $N_0 \geq r$ and
$ C(T, \gamma, r)$ s.t.
\begin{equation*}
\forall N\geq N_0, \quad  \E_{\Q^{r,N}} \exp\left\{ \gamma
\int_0^T|\beta^{(r)}_t(\widehat{X})|^2dt\right\}\leq C(T, \gamma, r).
\end{equation*}
\end{proposition}
%In order to control the exponential martingale, notice that there are essentially two types of estimates. Ones are as in the full Girsanov transformation used for existence and will involve Proposition \ref{est1}. However, as we are removing also the dependence of first $k$ particles from the rest, the following estimate will be used,

\begin{proof}
For $x\in C([0,T];\R)^N$, one has
%\jeff{,j\neq i} not necessary
\begin{align*}
|\beta^{(r)}_t(x)|^2&= \sum_{i=1}^r \left( \frac{1}{N}\sum^N_{{j=1}}
\int_0^t K_{t-s}(x_t^i-x_s^j)ds\1_{\{x_t^{j}\neq x^{i}_t\}}\right)^2\\
&\quad +
\frac{1}{N^2}\sum_{j=1}^{N-r}\left(\sum_{i=1}^r \int_0^t
K_{t-s}(x^{r+j}_t-x^{i}_s)ds\1_{\{x_t^{r+j}\neq x^{i}_t\}}\right)^2.
\end{align*}
%Old:
%$$|\beta^{(k)}_t(x)|^2= \sum_{i=1}^k \left( \frac{1}{N}\sum_{j=1}^N
%\int_0^t K_{t-s}(x_t^i-x_s^j)ds\right)^2+
%\frac{1}{N^2}\sum_{j=1}^{N-k}\left(\sum_{i=1}^k \int_0^t
%K_{t-s}(x^{k+j}_t-x^{i}_s)ds\right)^2. $$
By Jensen's inequality,
$$|\beta^{(r)}_t|^2 \leq \frac{1}{N}\sum_{i=1}^r \sum_{j=1}^N F_t(x^i,
x^j) + \frac{r}{N^2}\sum_{j=1}^{N-r}\sum_{i=1}^r F_t(x^{r+j}, x^i). $$
For simplicity we below write $\E$ (respectively, $\hat{X}^i$) instead
of
$\E_{\Q^{r,N}}$ (respectively, $\hat{X}^{i,N}$).
%again below drop the index $N$.
Observe that
%Applying Cauchy-Schwarz inequality and after applying H\"{o}lder's
%inequality,
\begin{align*}
& \E\exp\Big\{ \gamma
\int_0^T|\beta^{(r)}_t(\widehat{X})|^2dt\Big\}\\
& \leq  \Big( \E \exp\Big\{\sum_{i=1}^r \frac{2\gamma
}{N}\sum_{j=1}^N\int_0^T F_t(\widehat{X}^i,\widehat{X}^j)
dt\Big\}\Big)^{1/2}  \Big( \E \exp\Big\{  \frac{2\gamma  r}{N^2}
\sum_{j=1}^{N-r}\sum_{i=1}^r
\int_0^TF_t(\widehat{X}^{r+j},\widehat{X}^i)  dt\Big\}\Big)^{1/2}\\
%& \leq  \left( \prod_{i=1}^k \E_{\Q^{k,N}}\exp\left\{  \frac{2\gamma
%k}{N}\sum_{j=1}^N \int_0^T F_t(\widehat{X}^i,\widehat{X}^j)
%dt\right\}  \right)^{\frac{1}{2k}} \\
%&~~~~~~~~~~~~~~~~~~~~~~~~~~~~~~~~~~~~~~~~
%\times  \left( \prod_{j=1}^{N-k} \E_{\Q^{k,N}} \exp
%\left\{\frac{2\gamma  k(N-k)}{N^2}  \sum_{i=1}^k \int_0^T
%F_t(\widehat{X}^{k+j},\widehat{X}^i)  dt \right\}
%\right)^{\frac{1}{2(N-k)}} \\
%%\end{align*}
%%Therefore,
%%\begin{align*}
%%& \E_{\Q^{k,N}} \exp\left\{ \gamma
%%\int_0^T|\beta^{(k)}_t(\widehat{X})|^2dt\right\}
%%%
&\leq  \Big( \prod_{i=1}^r \frac{1}{N}
\sum_{j=1}^N\E\exp\Big\{2\gamma  r \int_0^T
F_t(\widehat{X}^i,\widehat{X}^j)  dt\Big\}  \Big)^{\frac{1}{2r}}
%\\
%&~~~~~~~~~~~~~~~~~~~~~~~~~~~~~~~~~~~~~~~~\times
\Big( \prod_{j=1}^{N-r} \frac{1}{r}\sum_{i=1}^r
\E \exp \Big\{\frac{2\gamma  r^2}{N}
\int_0^TF_t(\widehat{X}^{r+j},\widehat{X}^i)  dt \Big\}
\Big)^{\frac{1}{2(N-r)}}.
\end{align*}
In view of Proposition \ref{est1}, it now remains to prove that there exists $N_0 \in \N$ such that
\begin{equation*}
\sup_{N \geq N_0}\E \left[\exp \left\{\frac{2\gamma  r^2}{N}
\int_0^TF_t(\hat{X}^{r+j},\hat{X}^i)  dt \right\} \right]\leq C(T, r,
\gamma).
%Old: \sup_{N \geq N_0}\E_{\Q^{k,N}} \left[\exp \left\{\frac{2\gamma  k^2}{N}   \int_0^TF_t(\hat{X}^{k+j},\hat{X}^i)  dt \right\} \right]\leq C(T, k, \gamma).
\end{equation*}
We postpone the proof of this inequality  %Old 15 12: this proof
to the Appendix (see Proposition \ref{est2}).
%Finally, for $N \geq N_0$
%$$\E_{\Q^{k,N}} \exp\left\{ \gamma
%%%\int_0^T|\beta^{(k)}_t(\widehat{X})|^2dt\right\} \leq C(T, k,
%%%%\gamma)^{k \frac{1}{2k}} C(T, k, \gamma)^{(N-k) \frac{1}{2(N-k)}}=
%%%%%C(T, k, \gamma).$$
\end{proof}

\subsection{Propagation of chaos : Proof of Theorem \ref{PSPropaChaos}}
\modif{\subsubsection{Tightness}}
We start with showing the tightness of $\{\mu^{N}\}$ and of an auxiliary empirical measure which is needed in the sequel.
\modif{
\begin{lemma}
\label{Tightness:crit}
Let $\Q^N$ be as above.
The sequence $\{\mu^{N}\}$ is tight under $\Q^N$. In addition, let
$\nu^N:= \frac{1}{N^4}\sum_{i,j,k,l=1}^N
\delta_{X^{i,N}_.,X^{j,N}_.,X^{k,N}_.,X^{l,N}_.}$. The sequence
$\{\nu^{N}\}$ is tight under $\Q^N$.
\end{lemma}
}
\begin{proof}
The tightness of $\{\mu^{N}\}$,
respectively $\{\nu^{N}\}$, results from the tightness of the
intensity
measure $\{\E_{\Q^N} \mu^N(\cdot)\}$, respectively$\{\E_{\Q^N}
\nu^N(\cdot)\}$: See Sznitman~\cite[Prop. 2.2-ii]{Sznitman}. %Old: See Sznitman~\cite{Sznitman}.
By symmetry, in both cases it suffices to check the
tightness of $\{\text{Law}(X^{1,N})\}$. We aim to prove
\begin{equation}
\label{PS:KolmCrit}
\exists C>0, \forall N\geq N_0, \quad
\E_{\Q^N}[|X_t^{1,N}-X_s^{1,N}|^4] \leq C_T |t-s|^2, \quad 0\leq
s,t\leq T,
\end{equation}
where $N_0$ is as in Proposition \ref{GirsanovKparticles}. Let
$Z_T^{(1)}$ be as in~(\ref{def:Z^(k)}). One has
$$\E_{\Q^N}[|X_t^{1,N}-X_s^{1,N}|^4] =
\E_{\Q^{1,N}}[(Z_T^{(1)})^{-1}|\widehat{X}_t^{1,N}-\widehat{X}_s^{1,N}|^4].
$$
As $\widehat{X}^{1,N}$ is a one dimensional Brownian motion under
$\Q^{1,N}$,
%Applying Cauchy-Schwarz inequality,
$$\E_{\Q^N}[|X_t^{1,N}-X_s^{1,N}|^4]\leq
(\E_{\Q^{1,N}}[(Z_T^{(1)})^{-2}])^{1/2}
(\E_{\Q^{1,N}}[|\widehat{X}_t^{1,N}-\widehat{X}_s^{1,N}|^8])^{1/2}\leq
(\E_{\Q^{1,N}}[(Z_T^{(1)})^{-2}])^{1/2}  C |t-s|^2.$$
Observe that, for a Brownian motion $(W^\sharp)$ under $\Q^{1,N}$,
$$\E_{\Q^{1,N}}[(Z_T^{(1)})^{-2}]=\E_{\Q^{1,N}} \exp\left\{ 2
\int_0^T\beta_t^{(1)}(\hat{X})\cdot dW_t^\sharp - \int_0^T
|\beta^{(1)}_t(\hat{X})|^2dt\right\} .$$
Adding and subtracting  $\modif{3}\int_0^T|\beta^{(1)}_t|^2dt$ and applying
again the Cauchy-Schwarz inequality,
$$\E_{\Q^{1,N}}[(Z_T^{(1)})^{-2}]\leq \left(\E_{\Q^{1,N}}
\exp\left\{\modif{6}\int_0^T|\beta_t^{(1)}(\hat{X})|^2dt\right\}\right)^{1/2}.$$
Applying Proposition \ref{GirsanovKparticles} with $k=1$ and
$\gamma=\modif{6}$, we obtain the desired result.
%$$\E_{\P^N}[|X_t^{1,N}-X_s^{1,N}|^4] \leq C(T, \chi) |t-s|^2.$$
\end{proof}
\modif{\subsubsection{Convergence}}
%We now proceed to the proof of Theorem~\ref{PSPropaChaos}.
To prove Theorem %Old:  the \jeff{t}heorem
~\ref{PSPropaChaos} we have to show that any limit
point of $\{\text{Law}(\mu^N)\}$ is $\delta_\Q$, where $\Q$ is the
unique solution to (MPKS). Since the particles
%do not interact through a bounded kernel, but
interact through an unbounded singular functional,
%defined on the path space,
we adapt the arguments in Bossy and Talay \cite[Thm. 3.2]{BossyTalay}.

Let $\phi \in C_b(\R^p)$, $f\in C_b^2(\R)$, $0<t_1< \cdots <t_p\leq s
<t \leq T$ and $m\in \Pp(C[0,T];\R)$. Set
\begin{multline*}
G(m):= \int_{(C[0,T];\R)^2} \phi(x^1_{t_1}, \dots, x^1_{t_p})
\Big(f(x^1_t) -f(x^1_s) \\
-\modif{\frac{1}{2}}\int_s^tf''(x^1_u)du - \int_s^t
f'(x^1_u)\1_{\{x^1_u \neq x^2_u\}}\int_0^u
K_{u-\theta}(x^1_u-x^2_\theta)d\theta du\Big)dm(x^1)\otimes dm(x^2).
\end{multline*}
%Consider a subsequence of $\{\mu_N\}$, still denoted by $\{\mu_N\}$,
%which converges in law to %Old 15 12: towards to
%some probability measure $\Pi$.
%Let us show that the support of $\Pi$ is reduced to $\{\Q\}$.
%the set of solutions to (MPKS).
%As (MPKS) has a unique solution this will imply that $\mu_N$ converges
%in
%law towards $\Q$.

We start with showing that
\begin{equation}
\label{PoC:1}
\lim_{N \to \infty}\E [\left(G(\mu^N)\right)^2]=0.
\end{equation}
%Observe that
%\begin{align*}
%& G(\mu^N)= \frac{1}{N}\sum_{i=1}^N f(X^{i,N}_{t_1}, \dots,
%X^{i,N}_{t_k}) \times\\
%& \left(f(X^{i,N}_t) -f(X^{i,N}_s)-\int_s^tf''(X^{i,N}_u)du -
%\chi\int_s^t f^{'}(X^{i,N}_u)\frac{1}{N}\sum_{j\neq i }\1_{\{X^{i,N}_u
%\neq X^{j,N}_u\}}\int_0^u
%K_{u-\theta}(X^{i,N}_u-X^{j,N}_\theta)d\theta\,du\right).
%\end{align*}
%Old 18 12
%\begin{align*}
%& G(\mu^N)= \frac{1}{N}\sum_{i=1}^N f(X^{i,N}_{t_1}, \dots,
%X^{i,N}_{t_k}) \times\\
%& \left(f(X^{i,N}_t) -f(X^{i,N}_s)-\int_s^tf''(X^{i,N}_u)du -
%\chi\int_s^t f^{'}(X^{i,N}_u)\frac{1}{N}\sum_{j\neq i }\1_{\{X^{i,N}_u
%\neq X^{j,N}_u\}}\int_0^u
%K_{u-\theta}(X^{i,N}_u-X^{j,N}_\theta)d\theta\right).
%\end{align*}
Observe that
\begin{align*}
& G(\mu^N)= \frac{1}{N}\sum_{i=1}^N \phi(X^{i,N}_{t_1}, \dots, X^{i,N}_{t_p})  \Big(f(X_t^{i,N})-f(X_s^{i,N})- \modif{\frac{1}{2}}\int_s^tf''(X^{i,N}_{u})du   \\
&  -\frac{1}{N} \sum_{j=1}^N \int_s^t f'(X^{i,N}_u)\1_{\{X^{i,N}_u \neq
X^{j,N}_u\}}\int_0^u
K_{u-\theta}(X^{i,N}_u-X^{j,N}_\theta) d\theta \ du \Big).
\end{align*}
Apply It\^o's formula to $\frac{1}{N}\sum_{i=1}^N(f(X_t^{i,N})-f(X_s^{i,N}))$. It comes:
$$\E [\left(G(\mu^N)\right)^{\modif{2}}]\leq \frac{C}{N^2} \E\Big(\sum_{i=1}^N
\int_s^t
f'(X_u^{i,N})dW_u^i\Big)^2 \leq \frac{C}{N}. $$
%Set
%$\nu^N:= \frac{1}{N^4}\sum_{i,j,k,l=1}^N
%\delta_{X^{i,N}_.,X^{j,N}_.,X^{k,N}_.,X^{l,N}_.}$.
%As shown in ~\cite{Sznitman}, the tightness  of $\{\nu^N\}$
%results from
%the tightness of the intensity measures~$<I^N,f>:=\E <\nu^N,f>$. By
%symmetry, the latter reduces to the tightness of the laws of the
%$X^{1,N}$'s
%which was obtained in Proposition \ref{Tightness:crit}.
Thus, \eqref{PoC:1} holds true.

\modif{Suppose for a while we have proven the following lemma:
\begin{lemma}
\label{PoC:biglemma}
Let $\Pi^\infty \in \mathcal{P}(\mathcal{P}(C([0,T];\R)^4))$ be a limit point of $\{\text{law}(\nu^N)\}$.
%(remember Proposition \ref{Tightness:crit}).
Then
\begin{equation}
\label{PoC:BigLimit}
\begin{split}
& \lim_{N \to \infty}\E [\left(G(\mu^N)\right)^2]=
\int_{\mathcal{P}(C([0,T];\R)^4)} \left\{ \int_{C([0,T];\R)^4} \Big[
f(x^1_t)
-f(x^1_s)-\frac{1}{2}\int_s^tf^{''}(x^1_u)du \right. \\
&\left. ~~ - \int_s^t f^{'}(x^1_u)\1_{\{x^1_u
\neq x^2_u\}}\int_0^u K_{u-\theta}(x^1_u-x^2_\theta)d\theta du\Big]
\times \phi(x^1_{t_1}, \dots, x^1_{t_p}) d\nu(x^1\ldots,,x^4)
\vphantom{\frac15}\right\}^2 d\Pi^\infty(\nu),
\end{split}
\end{equation}
and
\begin{enumerate}[i)]
\item Any $\nu \in \mathcal{P}(C([0,T];\R)^4)$ belonging to the support of $\Pi^\infty $ is a product measure: $\nu = \nu^1 \otimes \nu^1 \otimes \nu^1
\otimes \nu^1$.
\item For any $t\in (0,T]$, the time marginal $\nu^1_t$ of $\nu^1$ has
a density $\rho_t^1$ which satisfies
$$\exists C_T, ~ \forall 0<t\leq T, ~ ~ \|\rho_t^1\|_{L^2(\R)}\leq \frac{C_T}{t^{\frac{1}{4}}}.$$
\end{enumerate}
\end{lemma}
}
%We take a non relabeled convergent subsequence of
%$\{\text{law}(\nu^N)\}$. Denote it's limit by $\Pi^\infty \in
%\mathcal{P}(\mathcal{P}(C([0,T];\R)^4))$.\\
%Our next goal is the following equality:
Then, combining \eqref{PoC:1} with the above result, we get
\begin{align*}
&\int_{C([0,T];\R)}  \phi(x^1_{t_1}, \dots, x^1_{t_p})\Big [ f(x_t^1)-f(x_s^1)- \frac{1}{2}\int_s^t f''(x_u)du \\
&  -  \int_s^t f'(x^1_{u}) \int_0^{u}\int K_{u-\theta}(x^1_{u}-y)\rho_\theta^1(y)dy d\theta du\Big]d\nu^1(x^1)=0.
\end{align*}
We deduce that $\nu^1$ solves (MPKS) and thus that $\nu^1=\Q$. As
by definition $\Pi^\infty$ is a limit point of $\text{Law}(\nu^N)$, it
follows that any limit point of $\text{Law}(\mu^N)$ is $\delta_\Q$,
which ends the proof.

\modif{\subsubsection{Proof of Lemma \ref{PoC:biglemma}}}
\paragraph*{\modif{Proof of~\eqref{PoC:BigLimit}: Step 1.}}
Notice that
%For some functionals $\Phi_2, \Phi_3 $ and $\Phi_4$ we have
%For some functionals $\varphi$ and $\Phi$ which we do not explicit here,
\begin{align}
\label{PoC:2}
& \E [\left(G(\mu^N)\right)^{\modif{2}}] =\modif{ \frac{1}{N^2}\E \sum_{i,k=1}^N \Phi_2(X^{i,N},X^{k,N})
+  \frac{1}{N^3}\E\sum_{i,k,l=1}^N \Phi_3(X^{i,N},X^{k,N},X^{l,N}) }\nonumber \\
& \modif{+ \frac{1}{N^3}\E\sum_{i,j,k=1}^N \Phi_3(X^{k,N},X^{i,N},X^{j,N}) +
\frac{1}{N^4} \E\sum_{i,j,k,l=1}^N \Phi_4(X^{i,N},X^{j,N},X^{k,N},
X^{l,N})},
\end{align}
\modif{where
\begin{align*}
& \Phi_2(X^{i,N},X^{k,N}) := \phi(X^{i,N}_{t_1}, \dots,
X^{i,N}_{t_p})~\phi(X^{k,N}_{t_1}, \dots, X^{k,N}_{t_p})\\
& \times  \Big(f(X_t^{i,N})-f(X_s^{i,N})-
\frac{1}{2}\int_s^tf''(X^{i,N}_{u})du\Big)\Big(f(X_t^{k,N})-f(X_s^{k,N})-
 \frac{1}{2}\int_s^tf''(X^{k,N}_{u})du\Big), \\
&  \Phi_3(X^{i,N},X^{k,N},X^{l,N}) := -\phi(X^{i,N}_{t_1}, \dots,
X^{i,N}_{t_p})~\phi(X^{k,N}_{t_1}, \dots, X^{k,N}_{t_p})\\
& \times \Big(f(X_t^{i,N})-f(X_s^{i,N})-
\frac{1}{2}\int_s^tf''(X^{i,N}_{u_1})du_1\Big)\int_s^t
f'(X^{k,N}_u)\1_{\{X^{k,N}_u \neq X^{l,N}_u\}}\int_0^u
K_{u-\theta}(X^{k,N}_u-X^{l,N}_\theta)~d\theta \ du, \\
%\end{align*}
%and
%\begin{align*}
& \Phi_4(X^{i,N},X^{j,N},X^{k,N}, X^{l,N}) :=
\phi(X^{i,N}_{t_1}, \dots,X^{i,N}_{t_p})
\phi(X^{k,N}_{t_1},\dots,X^{k,N}_{t_p})
\int_s^t \int_s^t\int_0^{u_1}\int_0^{u_2} f'(X^{i,N}_{u_1})
~f'(X^{k,N}_{u_2})  \\
& \times K_{u_1-\theta_1}(X^{i,N}_{u_1}-X^{j,N}_{\theta_1})
K_{u_2-\theta_2}(X^{k,N}_{u_2}-X^{l,N}_{\theta_2}) \1_{\{X^{i,N}_{u_1}\neq
X^{j,N}_{u_1}\}} \1_{\{X^{k,N}_{u_2}\neq X^{l,N}_{u_2}\}} d{\theta_1}\ d{\theta_2}\ d{u_1} \ d{u_2}.
\end{align*}}

\modif{Let $C_N$ be the last term in the r.h.s. of~\eqref{PoC:2}.
In Steps 2-4 below we prove that $C_N$ converges as $N \to \infty$
and we
identify its limit.}~
%We have:
%\begin{align*}
%& C_N:=  \frac{1}{N^4}\sum_{i,j,k,l=1}^N \int_s^t \int_s^t
%\int_0^{u_1}\int_0^{u_2}\E f'(X^{i,N}_{u_1})
%f'(X^{k,N}_{u_2}) \phi(X^{i,N}_{t_1}, \dots, X^{i,N}_{t_p})\phi(X^{k,N}_{t_1}, \dots,
%X^{k,N}_{t_p})  \\
%& \times K_{u_1-\theta_1}(X^{i,N}_{u_1}-X^{j,N}_{\theta_1})
%K_{u_2-\theta_2}(X^{k,N}_{u_2}-X^{l,N}_{\theta_2}) \1_{\{X^{i,N}_{u_1}\neq
%X^{j,N}_{u_1}\}} \1_{\{X^{k,N}_{u_2}\neq X^{l,N}_{u_2}\}} d{\theta_1}\ d{\theta_2}\ d{u_1} \ d{u_2}.
%\end{align*}
%d\P_{X_.^i,X_.^j,X_.^k,X_.^l}(x^1,x^2,x^3,x^4)
Define the function~$F$ on $\R^{2p+6}$ as
\begin{multline} \label{def:F}
F(x^1, \dots, x^{2p+6}) := \phi(x^{7},\dots,x^{p+6})~\phi(x^{p+7}, 
\dots, x^{2p+6})~f'(x^1)~f'(x^3) \\
%\quad \quad \quad \quad \quad \quad \quad \quad 
\times
K_{u_1-\theta_1}(x^1-x^2) K_{u_2-\theta_2}(x^3-x^4) \1_{\{x^1\neq x^5\}}
\1_{\{x^3\neq x^6\}} \1_{\{\theta_1< u_1\}}\1_{\{\theta_2< u_2\}}.
\end{multline}
%\jeff{We then have}~$C_N\jeff{\leq}\int_s^t \int_s^t\int_0^{u_1}\int_0^{u_2} A_N~d\theta_1\ d\theta_2\ du_1 \ du_2$ with
 We set $C_N=\int_s^t \int_s^t\int_0^{u_1}\int_0^{u_2} A_N~d\theta_1\ d\theta_2\ du_1 \ du_2$ with
\begin{equation*}
%\label{eq:def-AN}
A_N:=\frac{1}{N^4}\sum_{i,j,k,l=1}^N \E(F(X^{i,N}_{u_1},
X^{j,N}_{\theta_1},
X^{k,N}_{u_2},X^{l,N}_{\theta_2},X^{j,N}_{u_1},X^{l,N}_{u_2},X^{i,N}_{t_1},\dots,
 X^{i,N}_{t_p},X^{k,N}_{t_1}, \dots, X^{k,N}_{t_p})).
\end{equation*}
We now~aim to show that $A_N$ converges
pointwise \modif{(Step~2),}~that~$|A_N|$ is bounded from above by an
integrable function w.r.t.~$d{\theta_1}\ d{\theta_2}\ d{u_1} \ d{u_2}$
\modif{(Step~3)}, and finally to identify the limit of $C_N$ \modif{(Step~4)}.

\paragraph*{\modif{Proof of~\eqref{PoC:BigLimit}: Step 2.}}
Fix $u_1,u_2 \in [s,t]$ and $\theta_1 \in [0,u_1)$ and $\theta_2 \in
[0,u_2)$.  Define $\tau^N$ as
$$\tau^N:= \frac{1}{N^4}\sum_{i,j,k,l=1}^N \delta_{X^{i,N}_{u_1},
X^{j,N}_{\theta_1},
X^{k,N}_{u_2},X^{l,N}_{\theta_2},X^{j,N}_{u_1},X^{l,N}_{u_2},X^{i,N}_{t_1},\dots,
 X^{i,N}_{t_p},X^{k,N}_{t_1},\dots, X^{k,N}_{t_p}}.$$
Define the measure $\Q^N_{u_1,\theta_1,u_2,\theta_2,t_1, \dots,t_p}$
on  $\R^{2p+6}$ as
$ \Q^N_{u_1,\theta_1,u_2,\theta_2,t_1, \dots,t_p}(A)= \E(\tau^N(A))$.
The convergence of $\{\text{law}(\nu^N)\}$ implies the weak convergence
of $\Q^N_{u_1,\theta_1,u_2,\theta_2,t_1, \dots,t_p}$ to the measure on
$\R^{2p+6}$ defined by
%$\Q_{u_1,\theta_1,u_2,\theta_2,t_1, \dots,t_p}$,
\begin{align*}
&\Q_{u_1,\theta_1,u_2,\theta_2,t_1,
\dots,t_p}(A):=\int_{\mathcal{P}(C([0,T];\R)^4)} \int_{C([0,T];\R)^4}
\1_{A}(x^1_{u_1}, x^2_{\theta_1},
x^3_{u_2},x^4_{\theta_2},x^2_{u_1},x^4_{u_2}, x^1_{t_1},\dots,\\
&~~~~~~~~~~~~~~~~~~~~~~~~~~~~~~~~~~~~~~~~~~~~~~~~~~~~~~~~~~~~~~~~~~~
x^1_{t_p},x^3_{t_1},\dots, x^3_{t_p})d\nu(x^1,x^2,x^3,x^4)d\Pi^{\infty}(\nu).
\end{align*}
Let us show that this probability measure has an $L^2$-density w.r.t. the Lebesgue measure~ on $\R^{2p+6}$.
Let $h\in \modif{C_c}(\R^{2p+6})$. By weak convergence,
\begin{align*}
& \left| <\Q_{u_1,\theta_1,u_2,\theta_2,t_1, \dots,t_p},h> \right|\\
 &= \left| \lim_{N \to \infty} \frac{1}{N^4} \sum_{i,j,k,l=1}^N \E
 h(X^{i,N}_{u_1}, X^{j,N}_{\theta_1},
 X^{k,N}_{u_2},X^{l,N}_{\theta_2},X^{j,N}_{u_1},X^{l,N}_{u_2},
 X^{i,N}_{t_1},\dots, X^{i,N}_{t_p},X^{k,N}_{t_1},\dots,
 X^{k,N}_{t_p})\right|.
\end{align*}
When, in the preceding sum, at least two indices are equal, we bound the expectation by $\|h\|_\infty.$
When $i\neq j \neq k\neq l$, we apply Girsanov's transform in
Section \ref{sec:GirsTrK} with four particles
%$k=4$
and Proposition~
\ref{GirsanovKparticles}. This procedure leads to
%\jeff{(for $\hat{X}^N$ as in \eqref{PS_Kdrift})}
\begin{multline*}
\left| <\Q_{u_1,\theta_1,u_2,\theta_2,t_1, \dots,t_p},h> \right|\leq
\lim_{N \to \infty} \Big(\|h\|_\infty \frac{C}{N} \\
+  \frac{C_T}{N^4} \sum_{i\neq j \neq k \neq l} \left(\E
h^2(\hat{X}^{i,N}_{u_1}, \hat{X}^{j,N}_{\theta_1},
\hat{X}^{k,N}_{u_2},\hat{X}^{l,N}_{\theta_2},\hat{X}^{j,N}_{u_1},\hat{X}^{l,N}_{u_2},
\hat{X}^{i,N}_{t_1},\dots, \hat{X}^{i,N}_{t_p},\hat{X}^{k,N}_{t_1},\dots,
\hat{X}^{k,N}_{t_p})\right)^{1/2}\Big).
%\left| <\Q_{u_1,\theta_1,u_2,\theta_2,t_1, \dots,t_p},h> \right|\\
%\leq C_T \left(\E h^2(\hat{X}^{1,N}_{u_1}, \hat{X}^{2,N}_{\theta_1},
%\hat{X}^{3,N}_{u_2},\hat{X}^{4,N}_{\theta_2},\hat{X}^{2,N}_{u_1},\hat{X}^{4,N}_{u_2},
%\hat{X}^{1,N}_{t_1},\dots, \hat{X}^{1,N}_{t_p},\hat{X}^{3,N}_{t_1},\dots,
%\hat{X}^{3,N}_{t_p})\right)^{1/2}.
%Old:\leq\left(\E h^2(W^{i,N}_{u_1}, W^{j,N}_{\theta_1},
%W^{k,N}_{u_2},W^{l,N}_{\theta_2},W^{j,N}_{u_1},W^{l,N}_{u_2},
%W^{i,N}_{t_1},\dots, W^{i,N}_{t_p},W^{k,N}_{t_1},\dots,
%W^{k,N}_{t_p})\right)^{1/2}.
%\leq C_{u_1,u_2,\theta_1,\theta_2, t_1, \dots, t_p}
%\|h\|_{L^2(\R^{2k+6})}.
\end{multline*}
All the processes $\hat{X}^{i,N},\ldots,\hat{X}^{l,N}$ being
independent Brownian motions we deduce that
$$ \left| <\Q_{u_1,\theta_1,u_2,\theta_2,t_1, \dots,t_p},h> \right|\leq
C_{u_1,u_2,\theta_1,\theta_2, t_1, \dots, t_p}
\|h\|_{L^2(\R^{2p+6})}.
%Old: \lim_{N \to \infty} (C_{u_1,u_2,\theta_1,\theta_2, t_1, \dots, t_p}
%\|h\|_{L^2(\R^{2p+6})} +
%\|h\|_\infty(\frac{1}{N}+\frac{1}{N^2}+\frac{1}{N^3}) ).
$$
It follows from Riesz's representation~ theorem that $\Q_{u_1,\theta_1,u_2,\theta_2,t_1, \dots,t_p}$ has a density w.r.t. Lebesgue's measure in $L^2(\R^{2p+6})$. Therefore,
the functional~$F$ is continuous $\Q_{u_1,\theta_1,u_2,\theta_2,t_1,
\dots,t_p}$ - a.e. Since for any fixed $u_1,u_2 \in
[s,t]$ and
$\theta_1 \in [0,u_1)$, $\theta_2 \in [0,u_2)$, $F$ is also bounded we
have
\begin{align*}
%& \lim_{N\to \infty}\frac{1}{N^4}\sum_{i,j,k,l=1}^N
%\int_{C([0,T];\R)^4}f'(x^1_{u_1})
%f'(x^3_{u_2}) \phi(x^1_{t_1}, \dots, x^1_{t_p})\phi(x^3_{t_1}, \dots,
%x^3_{t_p})  \\
%& \times K_{u_1-\theta_1}(x^1_{u_1}-x^2_{\theta_1})
%K_{u_2-\theta_2}(x^3_{u_2}-x^4_{\theta_2}) \1_{\{x^1_{u_1}\neq
%x^2_{u_1}\}} \1_{\{x^2_{u_2}\neq x^4_{u_2}\}}
%d\P_{X_.^i,X_.^j,X_.^k,X_.^l}(x^1,x^2,x^3,x^4)\\
%&
\lim_{N\to \infty} A_N = <\Q_{u_1,\theta_1,u_2,\theta_2,t_1,
\dots,t_p}, F>.
\end{align*}
%Old 15 12:
%\begin{align*}
%& \frac{1}{N^4}\sum_{i,j,k,l=1}^N \int_{C([0,T];\R)^4}f'(x^1_{u_1})
%f'(x^3_{u_2}) \phi(x^1_{t_1}, \dots, x^1_{t_p})\phi(x^3_{t_1}, \dots,
%x^3_{t_p})  \\
%& \times K_{u_1-\theta_1}(x^1_{u_1}-x^2_{\theta_1})
%K_{u_2-\theta_2}(x^3_{u_2}-x^4_{\theta_2}) \1\{x^1_{u_1}\neq
%x^2_{u_1}\} \1\{x^2_{u_2}\neq x^4_{u_2}\}
%d\P_{X_.^i,X_.^j,X_.^k,X_.^l}(x^1,x^2,x^3,x^4)\\
%& \rightarrow <\Q_{u_1,\theta_1,u_2,\theta_2,t_1, \dots,t_p}, F>.
%\end{align*}

\paragraph*{\modif{Proof of~\eqref{PoC:BigLimit}: Step 3.}}
In view of the definition~(\ref{def:F}) of $F$ we may restrict 
ourselves to the case $i\neq j$ and
$k\neq l$. Use the Girsanov transforms from Section
\ref{sec:GirsTrK} with $r_{i,j,k,l}\in \{2,3,4\}$ according to 
the 
respective 
cases $(i=k,j=l)$, $(i=k,j\neq l)$, $(i\neq k,j\neq l)$, etc. 
Below we write $r$ instead of $r_{i,j,k,l}$.
By exchangeability it comes:
$$A_N=\Big| \frac{1}{N^4}\sum_{i\neq j,k\neq l}
\E_{\Q^{r,N}}(Z^{(r)}_TF(\cdots))\Big|\leq
\frac{1}{N^4} \sum_{i\neq j,k\neq l}
\left(\E_{\Q^{r,N}}(Z^{(r)}_T)^2\right)^{1/2}
\Big(\E_{\Q^{r,N}}(F^2(\cdots))\Big)^{1/2}.$$
By Proposition \ref{GirsanovKparticles}, $\E_{\Q^{r,N}}(Z^{(r)}_T)^2$
can be bounded uniformly w.r.t. $N$. As the functions $f$ and $\phi$
are bounded we deduce
$$\sqrt{\E_{\Q^{r,N}}(F^2(\cdots))} \leq C \1_{\{\theta_1< u_1\}}
\1_{\{\theta_2< u_2\}}
%Old: \1\{\theta_1< u_1\} \1\{\theta_2< u_2\}
\left(\E_{\Q^{r,N}}(K^2_{u_1-\theta_1}(W^i_{u_1}-W^j_{\theta_1})
K^2_{u_1-\theta_1}(W^k_{u_2}-W^l_{\theta_2}))
 \right)^{1/2}, $$
for $i\neq j$, $k\neq l$ and $r\equiv r_{i,j,k,l}$. In view of 
\eqref{l2normKernel},  for any $0<\theta<u<T$ we have
$$\left(\E_{\Q^{r,N}}(K^4_{u-\theta}(W^i_{u}-W^j_{\theta})\right)^{1/4}
\leq \frac{C}{u^{1/8}} \|K_{u-\theta}\|_{L^4(\R)}\leq
\frac{C}{u^{1/8}(u-\theta)^{7/8} }.$$
Therefore,
$$\left(\E_{\Q^{r,N}}(F^2(\cdots))\right)^{1/2}\leq C \frac{
\1_{\{\theta_1< u_1\}} \1_{\{\theta_2< u_2\}}}{u_1^{1/8}(u_1-\theta_1)^{7/8}
u_2^{1/8}(u_2-\theta_2)^{7/8} } .$$
We thus have obtained:
$$A_N \leq C \frac{  \1_{\{\theta_1< u_1\}} \1_{\{\theta_2<
u_2\}}}{u_1^{1/8}(u_1-\theta_1)^{7/8} u_2^{1/8}(u_2-\theta_2)^{7/8} }.$$
We remark that the r.h.s. belongs to $L^1((0,T)^4)$.

\paragraph*{\modif{Proof of~\eqref{PoC:BigLimit}: Step 4.}} Steps 2 and 3
allow us to conclude that
\begin{align*}
%& \lim_{N\to \infty}\int_s^t\int_s^t \int_s^t\int_s^t
%\frac{1}{N^4}\sum_{i,j,k,l=1}^N
%\int_{C([0,T];\R)^4}f'(x^1_{u_1}) f'(x^3_{u_2}) \phi(x^1_{t_1}, \dots,
%x^1_{t_p})\phi(x^3_{t_1}, \dots, x^3_{t_p})  \\
%& \times\1_{\{\theta_1< u_1\}}
%\1_{\{\theta_2< u_2\}}  K_{u_1-\theta_1}(x^1_{u_1}-x^2_{\theta_1})
%K_{u_2-\theta_2}(x^3_{u_2}-x^4_{\theta_2}) \1_{\{x^1_{u_1}\neq
%x^2_{u_1}\}} \1_{\{x^2_{u_2}\neq x^4_{u_2}\}}\\
%& \times d\P_{X_.^{i,N},X_.^{j,N},X_.^{k,N},X_.^{l,N}}(x^1,\ldots,x^4)
%d\theta_1d\theta_2du_1du_2\\
%&
 \lim_{N\to \infty} C_N = \int_s^t\int_s^t \int_s^t\int_s^t
<\Q_{u_1,\theta_1,u_2,\theta_2,t_1, \dots,t_p},F>
~d\theta_1d\theta_2du_1du_2.
\end{align*}
%Old 15 12:
%\begin{align*}
%& \int_s^t\int_s^t \int_s^t\int_s^t \frac{1}{N^4}\sum_{i,j,k,l=1}^N
%\int_{C([0,T];\R)^4}f'(x^1_{u_1}) f'(x^3_{u_2}) \phi(x^1_{t_1}, \dots,
%x^1_{t_p})\phi(x^3_{t_1}, \dots, x^3_{t_p})  \1\{\theta_1< u_1\}
%\1\{\theta_2< u_2\} \\
%& \times K_{u_1-\theta_1}(x^1_{u_1}-x^2_{\theta_1})
%K_{u_2-\theta_2}(x^3_{u_2}-x^4_{\theta_2}) \1\{x^1_{u_1}\neq
%x^2_{u_1}\} \1\{x^2_{u_2}\neq x^4_{u_2}\} \\
%&~~~~~~~~d\P_{X_.^{i,N},X_.^{j,N},X_.^{k,N},X_.^{l,N}}(x^1,x^2,x^3,x^4)
%d\theta_1d\theta_2du_1du_2\\
%&  \longrightarrow \int_s^t\int_s^t \int_s^t\int_s^t
%<\Q_{u_1,\theta_1,u_2,\theta_2,t_1, \dots,t_p},
%F>d\theta_1d\theta_2du_1du_2.
%\end{align*}
By definition of $\Q_{u_1,\theta_1,u_2,\theta_2,t_1, \dots,t_p}$ and
$F$ we thus have obtained that
\begin{align*}
\lim_{N\rightarrow \infty} C_N  =&\int_{P(C([0,T];\R)^4)}\int_s^t\int_s^t
\int_{C([0,T];\R)^4}  f'(x^1_{u_1}) f'(x^3_{u_2}) \phi(x^1_{t_1},
\dots, x^1_{t_p})\phi(x^3_{t_1}, \dots, x^3_{t_p})  \\
&\times \int_0^{u_1}\int_0^{u_2}
K_{u_1-\theta_1}(x^1_{u_1}-x^2_{\theta_1})
K_{u_2-\theta_2}(x^3_{u_2}-x^4_{\theta_2}) \1_{\{x^1_{u_1}\neq
x^2_{u_1}\}} \1_{\{x^3_{u_2}\neq x^4_{u_2}\}} \\
&~~~~~~~~d\nu(x^1,x^2,x^3,x^4)
~d\theta_1~d\theta_2~du_1~du_2~d\Pi^\infty(\nu).
\end{align*}
%Old 15 12
%\begin{align*}
%& C_N  \rightarrow \int_{P(C([0,T];\R)^4)}\int_s^t\int_s^t
%\int_{C([0,T];\R)^4} \chi^2 f'(x^1_{u_1}) f'(x^3_{u_2}) \phi(x^1_{t_1},
%\dots, x^1_{t_p})\phi(x^3_{t_1}, \dots, x^3_{t_p})  \\
%&\times \int_0^{u_1}\int_0^{u_2}
%K_{u_1-\theta_1}(x^1_{u_1}-x^2_{\theta_1})
%K_{u_2-\theta_2}(x^3_{u_2}-x^4_{\theta_2}) \1_{\{x^1_{u_1}\neq
%x^2_{u_1}\}} \1_{\{x^3_{u_2}\neq x^4_{u_2}\}} \\
%&~~~~~~~~d\nu(x^1,x^2,x^3,x^4)d\theta_1d\theta_2du_1du_2d\Pi^\infty(\nu).
%\end{align*}

%Old 15 12
%\begin{align*}
%& C_N  \longrightarrow
%\int_{\mathcal{P}(C([0,T];\R)^4)}\Big[\int_{C([0,T];\R)^2}\int_s^t
%\chi^2
%f'(x^1_{u}) \phi(x^1_{t_1}, \dots, x^1_{t_p})  \\
%&\times \int_0^{u} K_{u-\theta}(x^1_{u}-x^2_{\theta}) \1_{\{x^1_{u}\neq
%x^2_{u}\}} d\theta du~d\nu^1(x^1)\otimes
%d\nu^1(x^2)\Big]^2d\Pi^\infty(\nu).
%\end{align*}

A similar procedure is applied to the three other terms in
the r.h.s. of~\eqref{PoC:2}. Together with the preceding,
we obtain~\eqref{PoC:BigLimit}.

\paragraph*{Proof of~ i) and ii).}
Now, we prove the claims i) and ii) of Lemma~\ref{PoC:biglemma}.
\begin{enumerate}[i)]
\item For any measure $\nu \in \mathcal{P}(C([0,T];\R)^4)$, denote its first
marginal
by $\nu^1$. One easily gets
$ \Pi^\infty~~\text{a.e.},~~ \nu = \nu^1 \otimes \nu^1 \otimes \nu^1
\otimes \nu^1$
(see \cite[Lemma 3.3]{BossyTalay}).
\item Take $h\in \modif{C_c}(\R)$. Using similar arguments as in the above
Step~1, for any $0<t\leq T$ one has $\Pi^\infty(d \nu)$ a.e.,
\begin{align*}
&<\nu^1_t, h>= \lim_{N \to \infty}\E_{\Q^N} <\mu^N_t, h>= \lim_{N \to \infty}\E_{\Q^N}(h(X_t^{1,N}))=\lim_{N \to \infty}\E_{\Q^{1,N}}(Z_T^{(1)}h(W_t^{1,N}))\\
& \leq \frac{C}{t^{1/4}} \|h\|_{L^2(\R)}.
\end{align*}
\end{enumerate}

\section{Appendix}
\begin{proposition}
\label{est2}
Same assumptions as in Proposition \ref{est1}. There exists $N_0 \in \N$ depending only on $T$ and $\alpha$, such that
\begin{equation*}
\sup_{N \geq N_0} \E_{\P}\left[ \exp\left\{\frac{\alpha}{N} \int_0^T
\left(\int_0^t K_{t-s}(Y_t-w_s)ds\1_{\{w_t\neq Y_t\}}\right)^2dt
\right\}\right]\leq C(T, \alpha).
\end{equation*}
\end{proposition}
Compared to the proof of Proposition~\ref{GirsanovKparticles},
as $w$ and $Y$ exchanged places in the left-hand side, it
is not so obvious to use the independence of Brownian increments.
However, the weight $\frac{1}{N}$ enables us to skip the localization
part (see Lemmas \ref{MainLemmas1} and \ref{Mainlemmas3}).
 % fin de rouge
\begin{proof}
%[Proof of Proposition \ref{est2}  ]
Fix $N\in \N.$  Set $I:= \int_0^T \left(\int_0^t
K_{t-s}(Y_t-w_s) ds\right)^2 dt$.
%$I_k:= \left( \int_0^T \left(\int_0^t
%K_{t-s}(Y_t-w_s) ds\right)^2 dt\right)^k$.
%Writing the exponential as a sum and using Fatou's lemma,
%\begin{align*}
%&\E_{\P} \exp \left\{\frac{\alpha}{N}\int_0^T \left(\int_0^t
%K_{t-s}(Y_t-w_s) ds\right)^2 dt \right\}\leq \\
%& \sum_{k=0}^\infty \frac{\alpha^k}{(N)^k k!}\E_{\P} \left( \int_0^T
%\left(\int_0^t  K_{t-s}(Y_t-w_s) ds\right)^2 dt\right)^k=:
%\sum_{k=0}^\infty \frac{\alpha^k}{(N)^k k!} I_k.  \\
%\end{align*}
%By Cauchy-Schwarz inequality,
One has
\begin{align*}
I^{\modif{k}} &\leq C \left( \int_0^T \int_0^t \frac{ds}{(t-s)^{3/4}} \int_0^t
\frac{(Y_t-w_s)^2}{(t-s)^{9/4}} e^{- \frac{(Y_t-w_s)^2}{t-s}}
dsdt\right)^k \\
&\leq C  T^{k/4}  \left(\int_0^T\int_0^t
\frac{(Y_t-w_s)^2}{(t-s)^{9/4}} e^{- \frac{(Y_t-w_s)^2}{t-s}}
dsdt\right)^k.
\end{align*}
%By Fubini-Tonneli's theorem,
%$$I_k \leq 4^k T^{k/4} \left(\int_0^T\int_s^T
%\frac{(Y_t-w_s)^2}{(t-s)^{9/4}} e^{- \frac{(Y_t-w_s)^2}{t-s}}
%dtds\right)^k.$$
For $0\leq s<T$ and for $(\omega,\widehat{\omega}) \in C([0,T];\R) \times C([0,T];\R)$, %Old 15 12: $(\omega,\widehat{\omega}) \in C((0,T);\R) \times C((0,T);\R)$,
define the functional $H_s$ as
$$H_{s}(\omega,\widehat{\omega})=\int_s^T
\frac{(\omega_t-\widehat{\omega}_s)^2}{(t-s)^{9/4}} e^{-
\frac{(\omega_t-\widehat{\omega}_s)^2}{t-s}}  dt.$$
As the processes $Y$ and $w$ are independent,
\begin{equation*}
\E_{\P}\left(\int_0^T H_s(Y,w)ds\right)^k=\int_{C([0,T];\R)}\E_{\P}
\left(\int_0^T H_s(x,w)ds\right)^k\P^{Y}(dx).
%Old 15 12: \E_{\P}\left(\int_0^T
%H_s(Y,w)ds\right)^k=\int_{C(0,T)}\E_{\P}  \left(\int_0^T
%H_s(x,w)ds\right)^k\P^{Y}(dx).
\end{equation*}
As before we observe that, for any $x \in C([0,T];\R)$,
%$$\left(\int_0^T
%H_s(x,w)ds\right)^k=k!\int_0^TH_{s_1}(x,w)\int_{s_1}^T \cdots
%\int_{s_{k-1}}^T H_{s_k}(x,w)ds_k \dots ds_1,$$
%from which
$$\E_{\P}\left(\int_0^T
H_s(x,w)ds\right)^k=k!\E_{\P}\int_0^TH_{s_1}(x,w)\int_{s_1}^T \dots%Old: \cdots
\E_{\P}^{\mathcal{G}_{s_{k-1}}}\left(\int_{s_{k-1}}^T
H_{s_k}(x,w)ds_k\right) \dots ds_1.$$
Therefore,
%$w_{s_k}= w_{s_k}-w_{s_{k-1}} + w_{s_{k-1}}$,
\begin{align*}
&\E_{\P}^{\mathcal{G}_{s_{k-1}}}\left(\int_{s_{k-1}}^T
H_{s_k}(x,w)ds_k\right) = \int_{s_{k-1}}^T \int_{s_k}^T \int
\frac{(x_t- z - w_{s_{k-1}} )^2}{(t-s_k)^{9/4}}  e^{-\frac{(x_t- z -
w_{s_{k-1}} )^2}{t-s_k}} g_{s_k-s_{k-1}}(z)dz dt ds_k\\
& \leq \int_{s_{k-1}}^T  \frac{C}{\sqrt{s_k-s_{k-1}}}\int_{s_k}^T \frac{1}{(t-s_k)^{3/4}} \int z^2 e^{-z^2}dzdtds_k\leq C T^{1/4} \sqrt{T-s_{k-1}}\leq C T^{3/4}.
\end{align*}
Finally,
\begin{align*}
& \E_{\P}\left(\int_0^T H_s(x,w)ds\right)^k  \leq k!
CT^{3/4}\E_{\P}\left(\int_0^TH_{s_1}(x,w)\int_{s_1}^T \cdots
\int_{s_{k-2}}^T H_{s_{k-1}}(x,w)ds_{k-1} \dots ds_1\right).
\end{align*}
Repeat the previous procedure $k-2$ times and use
that the density of $w_{s_1}$ is bounded by
$\frac{C}{\sqrt{s_1}}$. It comes:
\begin{align*}
&\E_{\P}\left(\int_0^T H_s(x,w)ds\right)^k
\leq k! C^{k-1}T^{3(k-1)/4}
\E_{\P}\left(\int_0^TH_{s_1}(x,w) ds_1\right) \\
&\leq  k! C^{k-1}T^{3(k-1)/4} \E_\P\int_0^T
\frac{C}{\sqrt{s_1}} \int_{s_1}^T \int\frac{(x_t- w_{s_1}
)^2}{(t-s_1)^{9/4}}  e^{-\frac{(x_t- w_{s_1})^2}{t-s_1}}dx dt ds_1
\leq k! C^k T^{\modif{\frac{3k}{4}}}.
\end{align*}
%Old:
%\begin{align*}
%\E_{\P}\left(\int_0^T H_s(x,w)ds\right)^k
%&\leq k! C^{k-1}T^{3(k-1)/4}
%\E_{\P}\left(\int_0^TH_{s_1}(x,w) ds_1\right) \\
%&\leq  k! C^{k-1}T^{3(k-1)/4} \E_\P\int_0^T
%\frac{C}{\sqrt{s_1}} \int_{s_1}^T \int\frac{(x_t- w_{s_1}
%)^2}{(t-s_1)^{9/4}}  e^{-\frac{(x_t- w_{s_1})^2}{t-s_1}}dx dt ds_1 \\
%%leq   \int_0^T \frac{1}{\sqrt{s_1}}\int_{s_1}^T
%%\frac{dt}{(t-s_1)^{3/4}}ds_1
%%\E_{\P}\left(\int_0^T H_s(Y,w)ds\right)^k
%&\leq k! C^k T^{\modif{\frac{3k}{4}}}.
%\end{align*}
This implies that for any $M \geq 1$,
$$\E_{\P} \sum_{k=1}^M \frac{\alpha^k I^{\modif{k}}}{N^k k!} \leq
\sum_{k=1}^M\frac{\alpha^k C^k T^k}{N^k }.$$
Choose $N_0$ large enough to have $\frac{\alpha}{N_0} CT<1$.
To conclude, we apply Fatou's lemma.
%to $\sum_{k=1}^M \frac{\alpha^k I_k}{N^k k!}$.

% We come back to the beginning and for fixed $T,\alpha>0$ and $C>0$,
%$$\E_{\P} \exp \left\{\frac{\alpha}{N}\int_0^T \left(\int_0^t
%K_{t-s}(\bar{Y}^2_t-\bar{W}^1_s) ds\right)^2 dt \right\}\leq
%\sum_{k=0}^\infty \frac{\alpha^k}{(N)^k }  C^k T^k.$$
%The sum will converge if $\frac{\alpha C T}{N}<1$. As $N \to \infty$, there exists $N_0 \in \N$ such that this condition is satisfied for all $N\geq N_0$. Thus,
%$$\E_{\Q^{1,N}} \exp \left\{\frac{2\kappa \chi^2}{N}\int_0^T
%\left(\int_0^t  K_{t-s}(\bar{Y}^2_t-\bar{W}^1_s) ds\right)^2 dt
%\right\}\leq \frac{1}{1- \frac{\kappa \chi^2 C T}{N_0}}.$$
%Thus, the quantity of interest is uniformly bounded with respect to $N$ starting from some finite $N_0 \in N$. In the case $N < N_0$, we could still apply the procedure from Proposition \ref{est1} and bound it. Then we take the maximum between all these bounds (there is $N_0$ of them). Thus, we conclude that,
\end{proof}
%\end{appendices}

\noindent
\textbf{Acknowledgment}: For the first author, the article was prepared within the framework of a subsidy granted to the HSE by the Government of the Russian Federation for the implementation of the Global Competitiveness Program.
%\pagebreak
\footnotesize{
\bibliography{biblio}
}
\end{document}